\documentclass[12pt]{article}
\usepackage{a4}

\usepackage{amsthm}
\usepackage{amsmath}
\usepackage{graphicx}
\usepackage{enumitem}
\setlist{noitemsep}
\setlist[enumerate,1]{label=(\arabic*)}

\newtheorem{theorem}{Theorem}
\newtheorem{prop}[theorem]{Proposition}
\newtheorem{lemma}[theorem]{Lemma}

\newtheorem{conj}[theorem]{Conjecture}

\newcommand{\sm}{\setminus}

\newcommand{\FF}{\mathcal F}
\newcommand{\HH}{\mathcal H}
\newcommand{\II}{\mathcal I}
\newcommand{\XX}{\mathcal X}

\newcommand{\vl}{l\kern-0.035cm\char39\kern-0.03cm}

\title{\textbf{Long paths and toughness of \emph{k}-trees and chordal planar graphs}\footnote{The research
was supported by the project LO1506 of the Czech Ministry of Education, Youth and Sports
and by the project 17-04611S of the Czech Science Foundation.}
}%
\author{Adam Kabela\thanks{Department of Mathematics, Institute for Theoretical Computer Science,
and European Centre of Excellence NTIS,
University of West Bohemia, Pilsen, Czech Republic. Email: \texttt{kabela@ntis.zcu.cz}.}
}

\date{}

\begin{document}
\maketitle
\begin{abstract}
 We show that every $k$-tree of toughness greater than $\frac{k}{3}$ is Hamilton-connected for $k \geq 3$.
 (In particular, chordal planar graphs of toughness greater than $1$ are Hamilton-connected.)
 This improves the result of Broersma et al. (2007) and generalizes the result of B{\"o}hme et al. (1999).

 On the other hand, we present graphs whose longest paths are short.
 Namely, we construct $1$-tough chordal planar graphs and $1$-tough planar $3$-trees,
 and we show that the shortness exponent of the class is $0$,
 at most $\log_{30}{22}$, respectively.
 Both improve the bound of B{\"o}hme et al.
 Furthermore, the construction provides
 $k$-trees (for $k \geq 4$) of toughness greater than $1$.
 \end{abstract}


\section{Introduction}
We continue the study of Hamiltonicity and toughness
of $k$-trees following Broer\-sma et al.~\cite{ktree}
and of chordal planar graphs
following B{\"o}hme et al.~\cite{chpl}.

We recall that for a positive integer $k$, a \emph{$k$-tree} is either the graph $K_k$
(that is, the complete graph on $k$ vertices) 
or a graph containing a vertex whose neighbourhood induces $K_k$
and whose removal gives a $k$-tree.
Clearly, $k$-trees are chordal graphs.
We recall that the \emph{toughness} of a
graph $G$ is the minimum, taken over all separating sets $X$ of
vertices of $G$, of the ratio of $|X|$ to the number of components of
$G - X$.
The toughness of a complete graph is defined as being infinite.
We say that a graph is \emph{$t$-tough} if its toughness is at least
$t$.

In~\cite{ktree}, Broersma et al. showed that certain level of toughness
implies that a $k$-tree has a Hamilton cycle (see also~\cite{LZ, SW}).

\begin{theorem}\label{broersma}
Let $k \geq 2$.
Every $\frac{k+1}{3}$-tough $k$-tree (except for $K_2$) is Hamiltonian.
\end{theorem}

In the same paper, they constructed $1$-tough $k$-trees
which have no Hamilton cycle for every $k \geq 3$.

An older result considering toughness and Hamiltonicity in another 
subclass of chordal graphs is due to
B{\"o}hme et al.~\cite{chpl} who showed the following:

\begin{theorem}\label{bohme}
Every chordal planar graph (on at least $3$ vertices) of toughness greater than $1$ is Hamiltonian.
\end{theorem}

In~\cite{gerl}, Gerlach
generalized Theorem~\ref{bohme} for planar graphs whose separating cycles of length at least four have chords.
In this paper, we present a different generalization of Theorem~\ref{bohme}
which also improves the result of Theorem~\ref{broersma}. 

The mentioned results were motivated by the following conjecture stated by Chv\'{a}tal~\cite{chva}.

\begin{conj}\label{conj:ch}
  There exists $t$ such that every $t$-tough graph
  (on at least $3$ vertices) is Hamiltonian.
\end{conj}

Conjecture~\ref{conj:ch} remains open.
Partial results are known for some restricted classes of graphs; for instance, for
different subclasses of chordal graphs (see~\cite{ktree, chpl, all, spli, KKS}),
and for the class of chordal graphs itself (see~\cite{18} or~\cite{10}).
The best known lower bounds regarding Conjecture~\ref{conj:ch}
for chordal graphs and for general graphs were shown in~\cite{74}.
The study of toughness of graphs (and Conjecture~\ref{conj:ch} in particular)
is well-documented by a series of survey papers,
we refer the reader to~\cite{surv} (for more recent results, see~\cite{Broersma}).

In addition to the result of Theorem~\ref{bohme}, 
B{\"o}hme et al.~\cite{chpl} presented $1$-tough
chordal planar graphs whose longest cycles are relatively short
(compared to the number of vertices of the graph);
and using the notion of shortness exponent by Gr{\"u}nbaum and Walther~\cite{sef},
they argued the following: 

\begin{theorem}\label{89}
The shortness exponent of the class of $1$-tough chordal planar graphs is at most $\log_9 8$.
\end{theorem}


We recall that the \emph{shortness exponent} of a
class of graphs $\Gamma$ is the $\liminf$,
taken over all infinite sequences $G_n$ of non-isomorphic graphs of $\Gamma$
(for $n$ going to infinity),
of the logarithm of the length of a longest cycle in $G_n$
to base equal to the number of vertices of $G_n$.
For more results considering the shortness exponent,
see the survey~\cite{owens}.
To conclude this section, we mention that
by the combination of results of Moser and Moon~\cite{momo}
and Chen and Yu~\cite{3-conn},
the shortness exponent of the class of $3$-connected planar graphs equals $\log_3 2$.
%


\section{New results}
\label{s2}
We recall that a graph is \emph{Hamilton-connected} if for every pair of its vertices,
there is a Hamilton path between them.
Clearly, every Hamilton-connected graph (on at least $3$ vertices) is Hamiltonian.
Using a simple argument, we improve the result of Theorem~\ref{broersma} as follows.
(This also improves the result of~\cite{LZ}
since Hamilton-connected chordal graphs are, in fact, panconnected.)

\begin{theorem}\label{main}
Let $k \geq 3$. Every $k$-tree of toughness
greater than $\frac{k}{3}$ is Hamilton-connected.
Furthermore, every $1$-tough $2$-tree (except for $K_2$) is Hamiltonian.
\end{theorem}

The proof of Theorem~\ref{main} is given in Section~\ref{s3}.
We also show that under this toughness restriction a graph
is chordal planar if and only if it is a $3$-tree or $K_1$ or $K_2$
(see Lemma~\ref{coin}).
In particular, Theorem~\ref{main} implies that 
chordal planar graphs of toughness greater than $1$ are Hamilton-connected
(it generalizes the result of Theorem~\ref{bohme}).

In the other direction,
we present $1$-tough chordal planar graphs
and $1$-tough planar $3$-trees whose longest paths and cycles are relatively short.

In particular, for every $\varepsilon > 0$, there exists a $1$-tough chordal planar graph $G$
whose longest path has less than $|V(G)|^{\varepsilon}$ vertices. 
In Section~\ref{s6}, we note that such graphs
can be obtained by considering the square of particular trees.
Consequently, we adjust the result of Theorem~\ref{89} as follows:

\begin{theorem}\label{0}
The shortness exponent of the class of $1$-tough chordal planar graphs is $0$.
\end{theorem}

We remark that the graphs constructed in~\cite{chpl}
are $3$-connected, so the bound $\log_9 8$ of Theorem~\ref{89} also applies to the shortness exponent
of the class of $1$-tough planar $3$-trees (see Lemma~\ref{3}).
In Section~\ref{s7}, we use the standard construction for bounding the shortness exponent
(for more details regarding this construction, see for instance~\cite{owens} or~\cite{update}),
and we improve this bound by the following:

\begin{theorem}\label{79}
The shortness exponent of the class of $1$-tough planar $3$-trees is at most $\log_{30}{22}$.
\end{theorem}

In Section~\ref{s8}, we extend the used construction, and we remark that 
there are $k$-trees of toughness greater than $1$
whose longest paths are relatively short for every $k \geq 4$.
(Meanwhile, $3$-trees of toughness greater than $1$ are Hamilton-connected by Theorem~\ref{main}.)
This remark slightly improves the lower bound on toughness of non-Hamiltonian $k$-trees
presented in~\cite{ktree}, and contradicts the suggestion of~\cite{SW}.


\section{Tough enough \emph{k}-trees are Hamilton-connected}
\label{s3}
In this section, we prove Theorem~\ref{main}. 
Simply spoken, the proof is inductive;
we choose a vertex on a path and we extend the path using particular neighbours of this vertex.
 
For a vertex $v$, we let $N(v)$ denote its \emph{neighbourhood},
that is, the set of all vertices adjacent to $v$. 
We say a set $S \subseteq N(v)$ is a \emph{squeeze} by $v$ if
the following properties are satisfied for $S$ and $R = N(v) \sm S$.
\begin{itemize}
\item
$2 \geq |S| \geq 1$ and $|R| \geq 2$.
\item
Every vertex of $S$ is adjacent to at least $|R| - 1$ vertices of $R$,
and every vertex of $R$ is adjacent to at least $|S| - 1$ vertices of $S$.
\end{itemize}

The basic ingredient for applying the induction is the following:

\begin{lemma}\label{path}
Let $P$ be some set of vertices of a graph $G$
and let $x_1$, $x_2$ and $v$ be distinct vertices of $P$
and let $S$ be a squeeze by $v$.
If $G - S$ has a path between $x_1$ and $x_2$ whose vertex set is $P$,
then $G$ has such path whose vertex set is $P \cup S$.
\end{lemma}
\begin{proof}
We let $uv$ and $vw$ be the edges (incident with $v$) of the considered path in $G - S$.
We note that the graph induced by $\{u,v,w\} \cup S$ has a Hamilton path between $u$ and $w$.
Thus, we can extend the considered path into a path between $x_1$ and $x_2$
whose vertex set is $P \cup S$. 
\end{proof}

We recall that a vertex whose neighbourhood induces a complete graph is
called \emph{simplicial}.
For further reference, we state the following fact (shown, for instance, in~\cite{update}).

\begin{prop}\label{simpl}
Adding a simplicial vertex to a graph does not increase its toughness.
\end{prop}

By definition, $k$-trees can be viewed as graphs constructed
iteratively from $K_k$ by adding one new simplicial vertex of degree $k$ in each step.
We recall that a vertex adjacent to all vertices of a graph is called \emph{universal}.
Considering a non-universal vertex $v$ of a $k$-tree and the set $S$ of all its neighbours of degree $k$,
we say $v$ is a \emph{twig} if $N(v) \sm S$ induces $K_k$ and $|S| \geq 1$;
and we say $S$ is the \emph{bud} of this twig.
We note the following two facts:

\begin{lemma}\label{newTwig}
Let $k \geq 1$ and let $G$ and $G^+$ be $k$-trees
such that $G$ is obtained from $G^+$ by removing a simplicial vertex.
If $t$ is a twig in $G$ but not in $G^+$,
then a vertex of the bud of $t$ is a twig in $G^+$.
\end{lemma}
\begin{proof}
Since $t$ is a twig in $G$ but not in $G^+$,
there exists a vertex $t'$ adjacent to $t$ 
such that $t'$ has degree $k$ in $G$, and degree $k+1$ in $G^+$.
Clearly, $t'$ is a twig in~$G^+$.
\end{proof}

\begin{lemma}\label{twig}
Let $k \geq 1$ and let $G$ be a $k$-tree (on at least $k+3$ vertices) of toughness greater than $\frac{k}{3}$.
Then $G$ has a twig.
Furthermore, if $v$ is a twig of $G$ and $S$ is its bud,
then $G-S$ is a $k$-tree of toughness greater than $\frac{k}{3}$.
In addition, if $k \geq 2$, then $S$ is a squeeze by $v$.
\end{lemma}
\begin{proof}
We consider an iterative construction of $G$, and 
we let $T$ denote the $k$-tree on $k+3$ vertices 
which is obtained in the corresponding iteration of the construction.
Proposition~\ref{simpl} implies that the toughness of $T$ is at least the toughness of $G$,
and we observe that there exists only one $k$-tree on $k+3$ vertices 
of toughness greater than $\frac{k}{3}$ (for a fixed $k$).
We note that $T$ has a twig.
Thus, Lemma~\ref{newTwig} implies that $G$ has a twig.

We consider a twig $v$ in $G$ and its bud $S$, and we let $R = N(v) \sm S$.
Clearly, $G-S$ is a $k$-tree.
Furthermore, the toughness of $G-S$ is at least the toughness of $G$ (by Proposition~\ref{simpl}).

In addition, we note that every vertex of $S$ is adjacent to precisely $|R| - 1$ vertices of $R$.
Since $v$ is non-universal, the toughness of $G$ implies that no two vertices of
$S$ have the same neighbourhood. In particular, for $k = 2$, we have $|S| \leq 2$.
For $k \geq 3$, the same follows from the fact that $G - R - v$ has at least $|S| + 1$ components and $|R| = k$.
Clearly, if $k \geq 2$ then $|R| \geq 2$; and we conclude that $S$ is a squeeze by~$v$.
\end{proof}

We note that, with Lemmas~\ref{path} and~\ref{twig} on hand,
we can easily show Hamiltonicity of $k$-trees of toughness greater than~$\frac{k}{3}$.
(We remark that $2$-trees of toughness greater than $\frac{2}{3}$ are, in fact, $1$-tough.)

\begin{lemma}\label{ham}
Let $k \geq 2$. Every $k$-tree (except for $K_2$) of toughness
greater than $\frac{k}{3}$ is Hamiltonian.
\end{lemma}
\begin{proof}
We let $G$ be the considered $k$-tree,
and we let $n$ denote the number of its vertices.
Clearly, if $n \leq k + 2$, then $G$ is Hamiltonian.
We can assume that $n \geq k + 3$.
We suppose that the statement is satisfied for graphs on at most $n-1$ vertices,
and we show it for $G$.

By Lemma~\ref{twig},
$G$ has a twig $v$; and we let $S$ be the bud of $v$.
Furthermore, $G-S$ is a $k$-tree of toughness greater than~$\frac{k}{3}$.
(Clearly, $G-S$ is distinct from $K_2$.)
By the hypothesis, $G-S$ has a Hamilton cycle, and we view it
as a Hamilton path containing $v$ as an interior vertex.
By Lemmas~\ref{path} and~\ref{twig}, we can prolong this path
and obtain a Hamilton path in $G$ whose ends are adjacent, that is, a Hamilton cycle.
\end{proof}

Aiming for the Hamilton-connectedness, we shall need two additional ingredients
which are given by Lemma~\ref{twoGood} and Proposition~\ref{stave}.
For $k \geq 2$, a \emph{basic $3$-twig} is the graph obtained from $K_{k+1}$
by choosing its three different subgraphs $K_k$ and by adding one new simplicial vertex to each of them.
(For instance, the basic $3$-twig for $k = 3$ is the graph $B$ depicted in Figure~\ref{f:79}.)

\begin{lemma}\label{twoGood}
Let $k \geq 1$ and let $G$ be a $k$-tree (on at least $k+4$ vertices)
of toughness greater than $\frac{k}{3}$.
If $G$ is distinct from the basic $3$-twig,
then $G$ has two non-adjacent twigs (whose buds are disjoint).
\end{lemma}
\begin{proof}
We consider an iterative construction of $G$,
and we note that all $k$-trees obtained during the construction have
toughness greater than $\frac{k}{3}$ (by Proposition~\ref{simpl}).
We consider the $k$-tree on $k+4$ vertices, and 
we observe that either it is the basic $3$-twig
or it has two non-adjacent twigs.
(Clearly, the buds of non-adjacent twigs are disjoint.)
In particular, we can assume that $G$ has more than $k+4$ vertices.

Consequently, we note that the $k$-tree on $k+5$ vertices obtained during the construction
has two non-adjacent twigs.
Using Lemma~\ref{newTwig},
we conclude that $G$ has two non-adjacent twigs.
\end{proof}

In a graph $G$, we say a \emph{$\Theta$-spanner} between vertices $x_1$ and $x_2$
is a spanning subgraph of $G$ consisting of three paths with the same ends $x_1$, $x_2$
such that (except for the ends) these paths are mutually disjoint,
and each of them has at least one interior vertex.
We shall use $\Theta$-spanners to address the setting in which
the ends of the desired Hamilton path are the only twigs of a $k$-tree.
(We note that a similar idea appeared in~\cite{all}.)

\begin{prop}\label{stave}
Let $k \geq 3$ and let $G$ be a $k$-tree (distinct from $K_4$)
of toughness greater than $\frac{k}{3}$ 
and let $x_1$ and $x_2$ be distinct vertices of degree $k$. 
Then $G$ has a $\Theta$-spanner between $x_1$ and $x_2$.
\end{prop}
\begin{proof}
Clearly, $K_k$ has no vertex of degree $k$.
Furthermore, there exists only one $k$-tree
on $k+1$ vertices and one on $k+2$ vertices,
and only one $k$-tree on $k+3$ vertices has the required toughness
(for a fixed $k$).

Considering these $k$-trees, we note that the statement is satisfied for graphs on at most $k + 3$ vertices.
We let $n$ denote the number of vertices of $G$, and we assume that $n \geq k + 4$.
We suppose that the statement is satisfied for graphs on at most $n-1$ vertices,
and we show it for $G$.

Let us suppose that there is a twig $v$ and its bud $S$
such that neither $x_1$ nor $x_2$ belongs to $S$.
By Lemma~\ref{twig} and by the hypothesis,
we can consider a $\Theta$-spanner between $x_1$ and $x_2$ in $G - S$;
and we let $P$ be the set of vertices of one of the three paths
between $x_1$ and $x_2$ of this $\Theta$-spanner such that $v$ belongs to $P$.
By Lemmas~\ref{path} and~\ref{twig}, there is a path with the same ends whose vertex set is $P \cup S$.
Thus, $G$ has a $\Theta$-spanner between $x_1$ and $x_2$.

We assume that every twig is adjacent to $x_1$ or $x_2$.
By Lemma~\ref{twoGood}, we can assume that there is a twig $x'_1$ and its bud $S'$
such that $x_1$ belongs to $S'$ and $x_2$ does not.
Clearly, $x'_1$ has degree $k$ in $G - S'$.
We consider a $\Theta$-spanner $Y$ between $x'_1$ and $x_2$ in $G - S'$;
and we let $N$ denote the set of all vertices adjacent to $x'_1$ in $Y$.
We choose a vertex $y$ of $N$ such that $y$ is adjacent to $x_1$ in $G$.
Clearly, the subgraph of $Y$
induced by $N \cup \{ x'_1 \} \sm \{ y \}$ is a path,
and we apply Lemmas~\ref{path} and~\ref{twig} and extend this path by adding vertices of $S'$.
We consider the resulting path and the edge $x_1 y$, and we
extend the graph $Y - x'_1$ into a $\Theta$-spanner between $x_1$ and $x_2$ in $G$.
\end{proof}

Finally, we use the tools introduced in this section and prove Theorem~\ref{main}.

\begin{proof}[Proof of Theorem~\ref{main}]
For $k = 2$, the statement is satisfied by Lemma~\ref{ham}.
We assume that $k \geq 3$.
We let $G$ be a $k$-tree of toughness greater than $\frac{k}{3}$,
and we let $n$ denote the number of its vertices.
We note that if $n \leq k + 3$, then $G$ is Hamilton-connected;
so we can assume that $n \geq k + 4$.
We suppose that the statement is satisfied for graphs on at most $n-1$ vertices,
and we show it for $G$
(that is,
we show that for an arbitrary pair of vertices $x_1$ and $x_2$,
$G$ has a Hamilton path between $x_1$ and $x_2$).

Let us suppose that $G$ has a twig distinct from $x_1$ and $x_2$.
By Lemma~\ref{twoGood}, we can choose a twig $v$ 
such that $x_1$ does not belong to the bud $S$ of $v$.
In case $x_2$ belongs to $S$,
we consider a Hamilton path between $x_1$ and $v$ in $G - x_2$,
and we extend it by adding the edge $v x_2$.
In case neither $x_1$ nor $x_2$ belongs to $S$,
we consider a Hamilton path between $x_1$ and $x_2$ in $G - S$,
and we note that it can be extended into a desired path in $G$
(by Lemmas~\ref{path} and~\ref{twig}).

We assume that every twig of $G$ belongs to $\{ x_1, x_2 \}$.
By Lemma~\ref{twoGood}, we can assume that $x_1$ and $x_2$ are non-adjacent twigs
and the corresponding buds $S_1$ and $S_2$ are disjoint.
We consider the graph $G' = G - S_1 - S_2$.
We note that $G'$ is distinct from $K_4$ and $x_1$ and $x_2$ have degree $k$ in $G'$,
and that $G'$ is a $k$-tree of toughness greater than $\frac{k}{3}$ (by Lemma~\ref{twig}).

We consider a $\Theta$-spanner $Z$ between $x_1$ and $x_2$ in $G'$ given by Proposition~\ref{stave}.
Clearly, $Z$ forms three paths in $G' - x_1 - x_2$.
We note that we can join these paths
(using the adjacency of their ends and using the vertices of $S_1$ and $S_2$)
and obtain a Hamilton path from $S_1$ to $S_2$ in $G - x_1 - x_2$.
Thus, we get a Hamilton path between $x_1$ and $x_2$ in $G$. 
\end{proof}


%
To clarify the relation between Theorem~\ref{bohme} and the
case $k = 3$ of Theorem~\ref{main}, we note the following:

\begin{lemma}\label{coin}
A graph of toughness greater than $1$
is chordal planar
if and only if it is either a $3$-tree or $K_1$ or $K_2$.
\end{lemma}

For convenience, we include a short proof of Lemma~\ref{coin}.
We shall use the facts stated in Lemmas~\ref{3} and~\ref{emb}
(shown by Patil~\cite{patil} and by Markenzon et al.~\cite[Lemma~24]{mjp}, respectively).
We recall that a graph is \emph{$H$-free}
if it contains no copy of the graph $H$ as an induced subgraph.

\begin{lemma}\label{3}
Let $k \geq 1$.
A graph (distinct from $K_k$) is a $k$-tree if and only if it is $k$-connected chordal and $K_{k+2}$-free. 
\end{lemma}

\begin{lemma}\label{emb}
Let $G$ be a $3$-tree.
Then $G$ is planar if and only if
$G - C$ consists of at most two components for every set of vertices $C$ inducing $K_3$.
\end{lemma}

The combination of Lemmas~\ref{3} and~\ref{emb} gives the desired equivalence.

\begin{proof}[Proof of Lemma~\ref{coin}]
We consider a chordal planar (and thus $K_5$-free) graph.
By the assumption on toughness, the graph is either $3$-connected
or $K_1$ or $K_2$ or $K_3$, and we apply the case $k = 3$ of Lemma~\ref{3}.

For the other direction, we consider a $3$-tree of toughness greater than $1$.
We note that a removal of three vertices creates at most two components,
and we apply Lemma~\ref{emb}.
\end{proof}


\section{Long paths in $\mathbf{1}$-tough chordal planar graphs}
\label{s6}
In this section, we shall show the following: 

\begin{prop}\label{short2}
For every $n_0$, there exists a $1$-tough chordal planar graph on $n>n_0$
vertices whose longest cycle has $4\log_2 \frac{n+2}{3}$ vertices
and whose longest path has $2(\log_2 \frac{n+2}{3})^2 + 2$ vertices.
\end{prop}

In particular, the first part of Proposition~\ref{short2}
immediately implies the result of Theorem~\ref{0}.

\begin{proof}[Proof of Theorem~\ref{0}]
We consider an infinite sequence of non-isomorphic graphs
given by Proposition~\ref{short2}. 
We recall that a graph on $n$ vertices
belonging to this sequence has a longest cycle on $4\log_2 \frac{n+2}{3}$ vertices.
Consequently, the considered shortness exponent is at most 
$\displaystyle \lim_{n\to\infty} \log_n (4\log_2 \tfrac{n+2}{3}) = 0$.
\end{proof}

We recall that a tree is \emph{cubic} if 
every non-leaf vertex has degree $3$. 
In order to prove Proposition~\ref{short2},
we consider the square of `balanced' cubic trees,
and we combine several known facts
(recalled in Theorems~\ref{square},~\ref{t} and Propositions~\ref{c} and~\ref{p}).
We let $G^2$ denote the \emph{square} of a graph $G$, that is,
the graph on the same vertex set as $G$ in which
two vertices are adjacent if and only if their distance in $G$ is
either $1$ or $2$.
Studying squares of trees,
Neuman~\cite{Neuman} presented necessary and sufficient conditions 
for the existence of a Hamilton path between a given pair of vertices.
As a corollary,
the characterization of trees whose square has a Hamilton cycle (Hamilton path)
follows.
(Later, these results were also proven separately, see~\cite{hs, Gould}.)
We consider the trees depicted in Figure~\ref{f:FX},
and we recall these characterizations (see Theorem~\ref{square}).
Similarly as above, we recall that a graph is \emph{$\HH$-free}
if it contains no copy of a graph from the family $\HH$ as an induced subgraph. 

\begin{figure}[ht]
    \centering
    \includegraphics[scale=0.7]{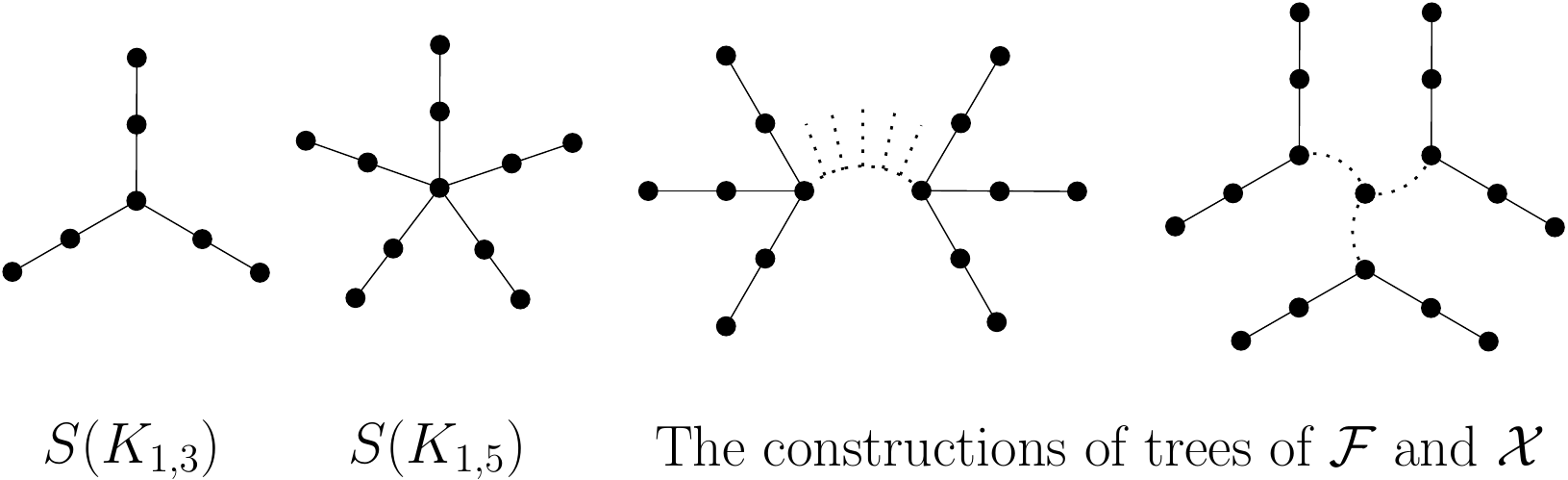}
    \caption{The trees $S(K_{1,3})$, $S(K_{1,5})$ and
    the families of trees $\FF$ and $\XX$.
    The trees of $\FF$ are obtained from two copies of $S(K_{1,3})$
    by joining their central vertices with a path (possibly an edge) 
    and adding one new vertex adjacent (by a pendant edge) to each interior vertex of this path.
    The trees of $\XX$ are obtained from three copies of $P_5$
    and from a tree containing precisely three leaves by identifying each of these
    leaves with the central vertex of one $P_5$.} 
    \label{f:FX}
\end{figure}

\begin{theorem}\label{square}
Let $T$ be a tree. The following statements are satisfied:
\begin{enumerate}[nosep]
\item $T^2$ is Hamiltonian if and only if $T$ (on at least $3$ vertices) is $S(K_{1,3})$-free.
\item $T^2$ has a Hamilton path if and only if $T$ is $S(K_{1,5})$-free,
$\FF$-free and $\XX$-free.
\end{enumerate}
\end{theorem}
%


In addition, we recall the following property of squares of graphs (shown by Chv\'{a}tal~\cite{chva}).

\begin{theorem}\label{t}
The square of a $k$-connected graph is $k$-tough.
\end{theorem}

We recall that (as observed by Fulkerson and Gross~\cite{elim})
a graph $G$ is chordal if and only if it has a \emph{perfect elimination ordering},
that is,
an ordering $(v_1, v_2, \dots, v_n)$ of all vertices of $G$ such that 
$v_i$ is a simplicial vertex of $G_i$ for every $i = 1, 2, \dots, n$,
where $G_i$ is the subgraph of $G$ induced by $\{ v_1, v_2, \dots, v_i \}$.
We note the following:

\begin{prop}\label{c}
The square of a tree is a chordal graph.
\end{prop}
\begin{proof}
Clearly, a perfect elimination ordering of the tree is a
perfect elimination ordering of its square. 
\end{proof}

We shall also use the following fact
(which we view as a corollary of the characterization of
graphs whose squares are planar by Harary et al.~\cite{hkw}). 

\begin{prop}\label{p}
Let $T$ be a tree.
Then $T^2$ is planar if and only if
$T$ has no vertex of degree greater than $3$.
\end{prop}
Finally, we construct graphs which have the properties stated in Proposition~\ref{short2}.

\begin{proof}[Proof of Proposition~\ref{short2}]
We let $T$ be a cubic tree (on at least $4$ vertices)
having a vertex such that the distances from this vertex
to every leaf are the same; and we let $r$ denote this distance. 
By Theorem~\ref{t} and Propositions~\ref{c} and~\ref{p},
$T^2$ is a $1$-tough chordal planar graph.

We let $n$ denote the number of vertices of $T$. 
By simple counting arguments, we get that $n = 3 \cdot 2^r - 2$
(that is, $r = \log_2 \frac{n+2}{3}$)
and that a largest $S(K_{1,3})$-free subtree of $T$ has $4r$ vertices.

Furthermore, $T$ is $S(K_{1,5})$-free and $\FF$-free
(since $T$ is a cubic tree).
We consider a largest $\XX$-free subtree, say $L$, and we show that it has $2r^2+2$ vertices.
We let $L_0$ be the tree obtained from $L$ by removing all leaves of $L$,
and we let $n_i$ be the number of vertices of degree $i$ in $L_0$ (for $i= 1,2,3$).
We note that all vertices of degree $3$ in $L_0$ belong to a common path
(since $L$ is $\XX$-free).
Hence, $n_3 \leq 2r-3$,
and therefore $n_2 \leq (r-2)^2$ and $n_1 \leq 2r - 1$.
Thus, $L$ has at most $n_3 + 2 n_2 + 3n_1 = 2r^2+2$ vertices
(that is, at most $n_3 + n_2 + n_1$ vertices of $L_0$ plus the removed leaves).
Lastly, we note that there is an $\XX$-free subtree of $T$ on $2r^2+2$ vertices.

We conclude that a longest cycle of $T^2$ has $4\log_2 \frac{n+2}{3}$ vertices
and its longest path has $2(\log_2 \frac{n+2}{3})^2 + 2$ vertices
by Theorem~\ref{square}.
\end{proof}


\section{Long paths in $\mathbf{1}$-tough planar $\mathbf{3}$-trees}
\label{s7}
In order to prove Theorem~\ref{79}, we show the following:

\begin{prop}\label{short79}
Let $n$ be a non-negative integer and
let $c(n) = 1 + 62(1 + 22 + \dots + 22^n)$.
Then there exists a $1$-tough planar $3$-tree $H_n$ on $1 + 70(1 + 30 + \dots + 30^n)$ vertices
whose longest cycle has $c(n)$ vertices and whose longest path has
$c(n) + 2 + 2( c(0) + c(1) + \dots + c(n-1))$ vertices.
\end{prop}

We note that the desired result follows as a corollary of Proposition~\ref{short79}.

\begin{proof}[Proof of Theorem~\ref{79}]
We consider the sequence of graphs $H_1, H_2, \dots$ given by Proposition~\ref{short79};
and for every $n \geq 0$, we let $f(n)$ denote the number of vertices of $H_n$. 
Clearly,
\begin{equation*}
f(n) = 1 + \tfrac{70}{29}(30^{n+1} - 1)
\quad
\textnormal{and}
\quad
c(n) = 1 + \tfrac{62}{21}(22^{n+1} - 1).
\end{equation*}
Thus,
\begin{equation*}
\lim_{n\to\infty} \log_{f(n)} c(n) = \log_{30}{22},
\end{equation*}
and therefore the considered shortness exponent is at most $\log_{30}{22}$.
\end{proof}

In the remainder of this section, we construct the graphs $H_n$ and prove Proposition~\ref{short79}.
We remark that, as well as in~\cite{chpl},
we shall use the standard construction for bounding the shortness exponent;
the improvement of the bound comes with a choice of a more suitable starting graph $H_0$.
The reasoning behind this choice is similar to the one applied in~\cite{update}.

\begin{figure}[h]
    \centering
    \includegraphics[scale=0.7]{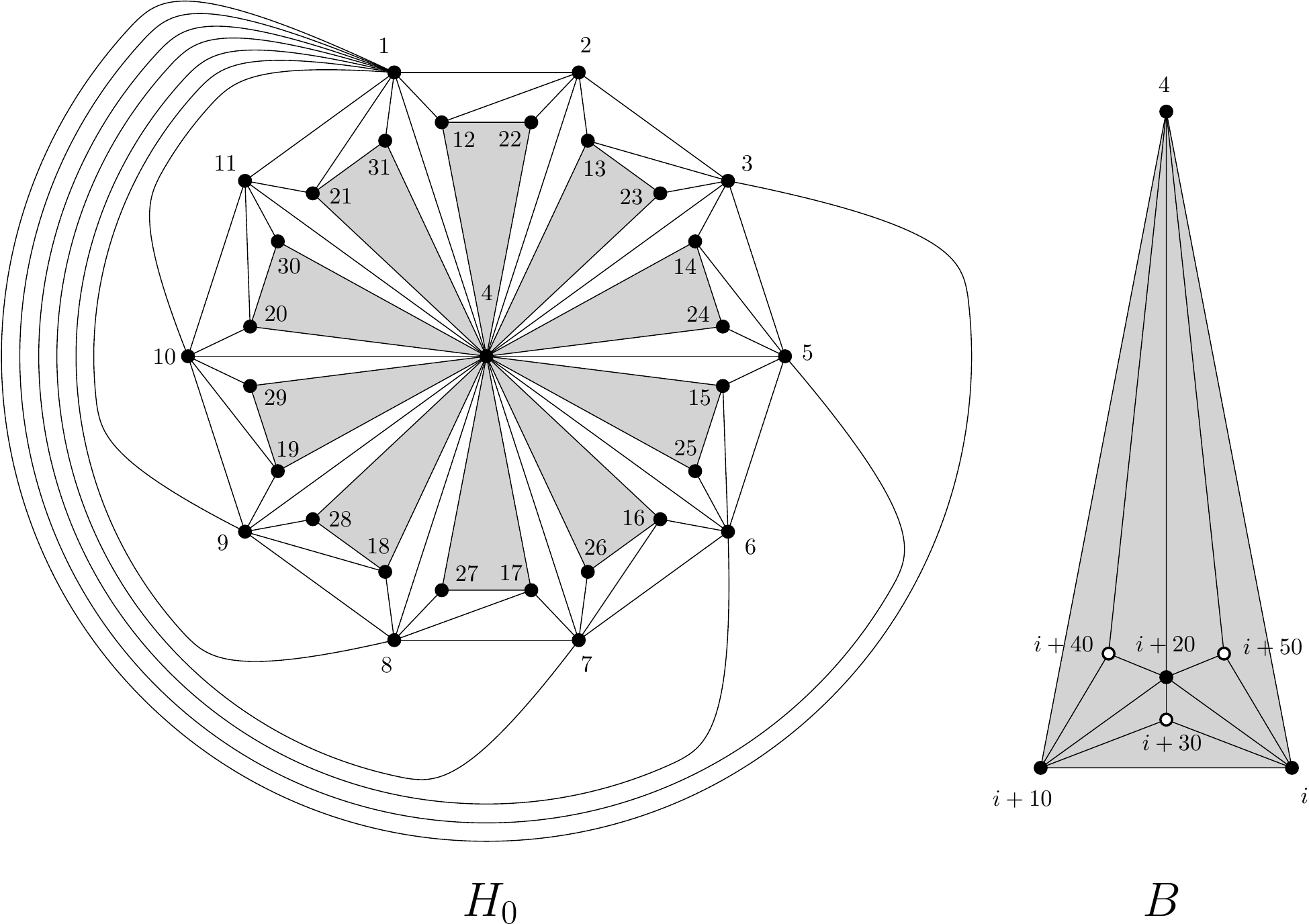}
    \caption{The graph $B$ and the construction of the graph $H_0$.
    The graph $H_0$ is obtained by replacing each of the highlighted triangles
    (of the graph depicted on the left)
    with a copy of $B$ in the natural way
    (by identifying the vertices of the
	highlighted triangle with the vertices of degree $5$ in $B$).
	The numbers represent the ordering of vertices of $H_0$.} 
    \label{f:79}
\end{figure}

We consider the graph $H_0$ constructed in Figure~\ref{f:79};
and we let $u_1, u_2, u_3$ denote the vertices of its outer face in the present embedding.
We note that $H_0$ contains $30$ vertices of degree $3$; and we call these vertices \emph{white}.

For every $n \geq 0$, we let $H_{n+1}$ be a graph obtained from $H_n$
by replacing every white vertex of $H_n$ with a copy of $H_0$
and by adding edges which connect the vertex $u_1, u_2, u_3$ of this copy
to precisely $1,2,3$ neighbours of the replaced vertex, respectively. 
We note the following:

\begin{prop}\label{3-trees}
For every $n \geq 0$, the graph $H_n$ is a planar $3$-tree.
\end{prop}
\begin{proof}
In accordance with the ordering suggested in Figure~\ref{f:79},
we let $u_1, u_2, \dots, u_{71}$ denote the vertices of $H_0$.

We show that the graphs $H_n$ are $3$-trees.
Clearly, $\{ u_1, u_2, u_3 \}$ induces $K_3$,
and we consider adding vertices $u_4, u_5, \dots, u_{71}$ in sequence (in this order),
and we observe that $H_0$ is a $3$-tree (by definition).

We view the replacement of a white vertex by a copy of $H_0$ as  
identifying this white vertex with the vertex $u_1$ of this copy and
adding vertices $u_2, u_3, \dots, u_{71}$ of this copy in sequence,
and we note that the resulting graph is a $3$-tree.
Consequently, $H_n$ is a $3$-tree for every $n \geq 0$.

We consider the planar embedding of $H_0$ given by Figure~\ref{f:79}.
When replacing a white vertex by a copy of $H_0$, we proceed in two steps.
First, we remove the white vertex, and we note that its neighbourhood induces a facial cycle.
Next, we embed a copy of $H_0$ inside this facial cycle, and we observe that the additional edges can be embedded as non-crossing.
We conclude that $H_n$ is planar for every $n \geq 0$.
(Alternatively, the planarity can be observed using Lemma~\ref{emb}.)
\end{proof}

To verify the toughness of the graphs $H_n$,
we shall use the following lemma (shown in~\cite{update}).

\begin{lemma}\label{constr2}
For $i = 1,2$, let $G^+_i$ and $G_i$ be $t$-tough graphs
such that $G_i$ is obtained by removing vertex $v_i$ from $G^+_i$.
Let $U$ be a graph obtained from the disjoint union of $G_1$ and $G_2$ by adding new edges 
such that the minimum degree of the bipartite graph $(N(v_1), N(v_2))$ is at least $t$.
Then $U$ is $t$-tough.
\end{lemma}

In order to apply Lemma~\ref{constr2}, we determine the toughness 
of $H^+_0$, that is, the graph obtained from $H_0$ by adding
one auxiliary vertex $x$ adjacent to $u_1, u_2$ and~$u_3$.

\begin{prop}\label{1-tough}
The graphs $H^+_0$ and $H_0$ are $1$-tough.
\end{prop}
\begin{proof}
We consider a separating set $S$ of vertices of $H^+_0$.
If $u_4$ belongs to a component of $H^+_0 - S$,
then every other component has precisely one vertex,
and we note that $|S| > c(H^+_0 - S)$.

We assume that $u_4$ belongs to $S$.
Except for $u_4$, the vertices adjacent to a white vertex are called \emph{black}.
Except for $u_4$ and $x$, the non-white and non-black vertices are called \emph{blue}.
We consider the set consisting of all white vertices and all black vertices which have no blue neighbour,
and we let $\II$ denote the set of all components of $H^+_0 - S$
whose every vertex belongs to the considered set.

We shall use a discharging argument.
We assign charge $1$ to every component of $H^+_0 - S$,
and we distribute all assigned charge among the vertices of $S$
according to the following rules.
\begin{itemize}
\item The component containing $x$ (if there is such) gives its charge to $u_4$.
\item The total charge of all components of $\II$
is distributed equally among black vertices of $S$.
\item The total charge of all remaining components is distributed 
equally among blue vertices of $S$.
\end{itemize}

We observe that every vertex of $S$ receives charge at most $1$,
that is, $|S| \geq c(H^+_0 - S)$.
Thus, $H^+_0$ is $1$-tough. Consequently, $H_0$ is $1$-tough by Proposition~\ref{simpl}.
\end{proof}

\begin{prop}\label{1-toughHn}
For every $n \geq 0$, the graph $H_n$ is $1$-tough.
\end{prop}
\begin{proof}
By Proposition~\ref{1-tough}, $H^+_0$ and $H_0$ are $1$-tough.
We consider an iterative construction of $H_n$
(replacing white vertices by copies of $H_0$ in sequence).
We shall apply Lemma~\ref{constr2}.
The graph at a current iteration plays the role of $G^+_1$ and the replaced vertex the role of $v_1$,
and $H^+_0$ and $H_0$ play the role of $G^+_2$ and $G_2$.
Using Lemma~\ref{constr2} repeatedly, 
we note that in each step of the construction we obtain a $1$-tough graph.
We conclude that $H_n$ is $1$-tough.
\end{proof}

We recall the standard construction for bounding the shortness exponent
(this construction produces graphs whose longest cycles are relatively short).
The idea of the construction is formalized in the following definition
and in Lemma~\ref{cyc} (which was proven in~\cite{update}).

An \emph{arranged block} is a $5$-tuple $(G_0, j, W, O, k)$
where $G_0$ is a graph, $j$ is the number of vertices of $G_0$,
and $W$ and $O$ are disjoint sets of vertices of $G_0$
such that the vertices of $W$ are simplicial and independent 
and $O$ induces a complete graph
and such that every cycle in $G_0$ contains at most $k$ vertices of $W$.

\begin{lemma}\label{cyc}
Let $(G_0, j, W, O, k)$ be an arranged block such that $k \geq 1$. 
For every $n \geq 1$, let $G_n$ be a graph obtained from $G_{n-1}$
by replacing every vertex of $W$ with a copy of $G_0$ (which contains $W$ and $O$),
and by adding arbitrary edges which connect the neighbourhood of the replaced vertex
with the set $O$ of the copy of $G_0$ replacing this vertex.
Then $G_n$ has $1 + (j - 1)(1 + |W| + \dots + |W|^n)$ vertices and its longest cycle has at most
$1 + (\ell - 1)(1 + k + \dots + k^n)$
vertices where $\ell = j - |W| + k$.
\end{lemma}

Finally, we show that the constructed graphs $H_n$ have all properties
stated in Proposition~\ref{short79}.

\begin{proof}[Proof of Proposition~\ref{short79}]
By Propositions~\ref{3-trees} and~\ref{1-toughHn},
$H_n$ is a $1$-tough planar $3$-tree (for every $n \geq 0$).
By a simple counting argument,
we get that $H_n$ has $1 + 70(1 + 30 + \dots + 30^n)$ vertices.

We observe that a path in $H_0$ contains at most $22 + z$ white vertices
where $z$ is the number of white ends of the path.
In particular, every cycle in $H_0$ contains at most $22$ white vertices.
By Lemma~\ref{cyc}, a longest cycle in $H_n$ has at most $c(n)$ vertices.

We let $p(n) = c(n) + 2 + 2( c(0) + c(1) + \dots + c(n-1))$ and $w(n) = 22^{n+1} + 2(1 + 22 + \dots + 22^n)$.
For the sake of induction, we show that every path in $H_n$ has at most
$p(n)$ vertices, and furthermore that it contains at most $w(n)$ white vertices
(a similar idea was used in~\cite{update}).
We note that the claim is satisfied for $n = 0$, and we proceed by induction on $n$.

We let $P$ be a path in $H_n$, and we consider suppressing vertices of $P$ as follows.
For every newly added copy of $H_0$,
we suppress all but one vertex of the copy
and we replace the remaining vertex (if there is such) by the corresponding replaced vertex of $H_{n-1}$;
and we let $P'$ be the resulting graph.
Since the neighbourhood of every replaced vertex induces a complete graph,
$P'$ is a path;
and we view $P'$ as a path in $H_{n-1}$.
By the hypothesis, $P'$ contains at most $w(n-1)$ white vertices.
Thus, $P$ visits at most $w(n-1)$ of the newly added copies of $H_0$.

Similarly, we choose an arbitrary newly added copy of $H_0$, and
we suppress all vertices of $P$ not belonging to this copy.
Since $\{ u_1, u_2, u_3 \}$ induces a complete graph, the resulting graph 
is a path in $H_0$ (possibly empty or trivial).
Considering such paths for all newly added copies of $H_0$,
and considering the set of all their ends, we note that at most two white vertices belong to this set.
Hence, in total these paths contain at most $63 \cdot w(n-1) + 2$ vertices.
We note that
\begin{equation*}
p(n) = p(n-1) - w(n-1) + 63 \cdot w(n-1) + 2.
\end{equation*}
Thus, $P$ has at most $p(n)$ vertices.
Furthermore, we note that $P$ contains at most $w(n) = 22 \cdot w(n-1) + 2$ white vertices.

To conclude the proof, we extend the earlier observation as follows.
In fact, there are paths in $H_0$ containing $22 + z$ white and all non-white vertices
such that all non-white ends belong to $\{ u_1, u_2\}$.
Using these paths,
we observe that $H_n$ has a cycle on $c(n)$ vertices
and a path on $p(n)$ vertices.
\end{proof}


\section{On \emph{k}-trees of toughness greater than one}
\label{s8}
To conclude the paper,
we remark that for every $k \geq 4$, there are 
$k$-trees of toughness greater than $1$ whose longest paths are relatively short.
For brevity, we omit enumerating the exact length of these paths.

We consider the $1$-tough $3$-trees $H_n$ given by Proposition~\ref{short79}.
Clearly, adding a universal vertex to a $k$-tree gives a $(k+1)$-tree.
For every $k \geq 4$ and every $n \geq 0$, we let $H_{n,k}$
denote the graph obtained by adding $k-3$ universal vertices to $H_n$;
and we note that $H_{n,k}$ is a $k$-tree of toughness greater than $1$.

We consider a path in $H_{n,k}$.
We remove the universal vertices of $H_{n,k}$ from this path,
and we view the resulting forest (whose components are paths) as a subgraph of $H_n$.  
By Proposition~\ref{short79}, every path of this forest is relatively short.
Consequently, we observe that for every $k \geq 4$, there exists
$n_0$ such that if $n \geq n_0$, then a longest path in $H_{n,k}$ is relatively short.
(We note that the same idea can be applied to the graphs constructed in~\cite{chpl}.)

\section*{Acknowledgement}

The author would like to thank Jakub Teska for his mentorship and for inspiring discussions  
(which led to a weaker version of Theorem~\ref{0}) which partly motivated this study,
and to thank the anonymous referees for their helpful suggestions and comments.


\end{document}